\documentclass[eqnumis, pdf, published]{hjms}
\usepackage{amsfonts}
\usepackage{amsmath}
\usepackage{footnote}
\usepackage{multicol}
\usepackage{booktabs}
\usepackage[flushleft]{threeparttable}
\usepackage[font=small,labelfont=bf]{caption}

\begin{document}
%

\markboth{S.Y. Su, C.L. Tong, Y.S. Yang}{Total colouring  of  $C_{n}(1, 3)$}

\title{Total colouring  of  circulant graphs $C_{n}(1, 3)$}

\author{SenYuan Su$^1$, Chunling Tong\coraut$^{1,2}$, Yuansheng Yang$^3$}

\address{$^1$School of Information Science and Electricity Engineering, Shandong Jiaotong University,  Jinan $250357$,  China\\
    $^2$Department of Mathematics, Simon Fraser University, Burnaby $V5A~1S6$, Canada\\
	$^3$College of Computer Science, Dalian University of Technology, Dalian, $116024$, China}
\emails{susy@stu.sdjtu.edu.cn (S.Y. Su), tongcl@sdjtu.edu.cn (C.L. Tong), \\ yangys@dlut.edu.cn (Y.S. Yang)}
\maketitle
\begin{abstract}
	Total colouring of 4-regular circulant graphs is an interesting but challenging topic, and has attracted much attention. However, it still remains an open question to determine the total chromatic numbers of $C_{n}(1, 3)$, a subclass of 4-regular circulant graphs, even after many efforts. In this paper, we investigate the total colouring of these graphs and  determine their total  chromatic numbers. Our results show that the total chromatic numbers of $C_{n}(1, 3)$ are 6 for $n=7,8,12,13,17$, and 5 for all others.
\end{abstract}



\section{Introduction}
We use $V(G)$ and $E(G)$ to denote the vertex set and edge set of  a simple connected graph $G$, respectively. A \textit{total $k$-colouring} of a graph $G$ is a map $\sigma$: $V(G)\cup E(G)\rightarrow\{1,2,\cdots,k\}$, such that no  two adjacent or incident elements of $V(G)\cup E(G)$ receive the same integers.  The smallest integer $k$ needed for such a map  is known as the \textit{total chromatic number}, denoted as $\chi''(G)$. Determining total chromatic number is NP-complete \cite{SArr89}, and NP-hard even for $r$-regular bipartite graphs with $r \geq 3$ \cite{Mc94}.

 For  a simple graph, there is a long-standing conjecture on total colouring, proposed by Behzad \cite{Be65} and Vizing \cite{Viz68} independently. It supposes that the total chromatic number of a simple graph $G$ is not more than $\Delta(G)+2$, where $\Delta(G)$ represents the maximum degree of $G$. The conjecture implies that for every simple graph $G$, the total chromatic number is either equal to $\Delta(G) + 1$ or equal to $\Delta(G) + 2$, since it cannot be less than $\Delta(G) + 1 $. Usually, a graph with $\chi''(G)$ = $\Delta(G)$ + 1 is known as Type I while  a graph with $\chi''(G)$ = $\Delta(G)$ + 2 is known as Type II. The conjecture has been verified by a great number of graphs, and exact values of total chromatic number for many graphs were determined \cite{Gee21,Gee23,Pal23,Pra22,Tong09,Tong19,Yap96}.

 A general 4-regular circulant graph, denoted as $C_n(d_1, d_2)$, is  the graph that has a vertex set $V$ = $\left\{v_0, v_1,\cdots,v_{n-1}\right\}$ and an edge set $E$ = $\bigcup_{i=1}^{2}E_i$ with $E_i = \{e_0^{i}, e_1^{i},\cdots, e_{n-1}^{i}\}$ and $e_m^{i} = v_mv_{m+d_i}$, where $1 \le d_1 < d_2 \le \left\lfloor{\frac{n-1}{2}}\right\rfloor$ and the indices of vertices are considered modulo $n$. A lot of works have contributed to determining the total chromatic numbers of 4-regular circulant graphs \cite{Cam03,Far23,Khe08,Nav22,Nig21}.

$C_n(1, 3)$  is a subclass of 4-regular circulant graphs with $d_1=1$ and $d_2=3$. Several studies have been conducted on the total colouring of  $C_{n}(1, 3)$\cite{Far23,Khe08,Nav22,Nig21}. Khennoufa and Togni conjectured that only a finite number of $C_{n}(1, 3)$ are Type II, while all others are Type I \cite{Khe08}. It has been proven that 
$C_{5p}(1,3)$ are Type I \cite{Khe08,Nav22}, and $C_{3p}(1,3)$ are  Type I except for  $C_{12}(1,3)$, which is Type II \cite{Far23,Nig21}. Despite these advances, determining the total chromatic numbers of $C_{n}(1, 3)$
 remains an open question for most  $n$.
  Given the limited  research findings, this paper  focuses on the  total colouring of $C_{n}(1, 3)$. We aim to find the total chromatic numbers of $C_n(1, 3)$ for all $n\geq 7$.

\section{Preliminaries }\label{Secpre}

We first recall two useful results obtained by Kostochka \cite{Kos77}, Chetwynd and Hilton \cite{Che88}, respectively. For our purpose, they can be recapitulated as the following lemmas.
\begin{lemma}
A 4-regular graph always has  a total 6-colouring \cite{Kos77}.
\label{Lem1T}
\end{lemma}

\begin{lemma}
If a regular graph $G$ has a total $(\Delta(G) + 1)$-colouring,  then it has a  vertex-colouring with colours $1,2,\dots,\Delta(G) + 1$ such that $|V_{j}|\equiv ~|V(G)|~(mod~2)$ $( 1\leq j\leq  \Delta(G) + 1 )$, where  $V_{j}$ is the set of vertices  coloured $j$ \cite{Che88}.
\label{Lem1}
\end{lemma}

The two lemmas will be helpful to our demonstration in determining the total chromatic numbers of $C_n(1, 3)$. According to Lemma \ref{Lem1T}, the total chromatic numbers of $C_n(1, 3)$ should be equal to either $5$ or 6. Therefore, Lemma \ref{Lem1T} tells us that  $\chi''(C_{n}(1,3)) = 6$ if and only if $\chi''(C_{n}(1,3))\neq 5$, while Lemma \ref{Lem1} can help us to  exclude the possibility of $\chi''(C_{n}(1,3))=5$ for some graphs.

As mentioned above, it is a challenging issue to determine  $\chi''(C_{n}(1,3))$ for all $n\geq 7$. We need to resolve the issue  case by case for different $n$.  Noting that any positive integer $n$ can be always expressed as $n=5k+1, 5k+2,~5k+3, ~5k+4$,~or $5k$ with $k$ being a nonnegative integers, we can further rewritten them as  $n=5(k-7)+9\times4,~5(k-5)+9\times3,~5(k-3)+9\times2,~5(k-1)+9\times 1$,~or~ $5k+9\times0$. Therefore,  $C_{n}(1, 3)$  for $n\geq 7$ can be classified into four distinct cases:  $C_n(1, 3)$  for $n = 5p + 9q$ with $p$ and $q$ being nonnegative integers, $C_n(1, 3)$ for $n=11,16,21,26,31$, $C_n(1, 3)$ for  $n=7,12,17,22$, and $C_n(1, 3)$  for $n=8,13$. For some of them, we can directly obtain $\chi''(C_{n}(1,3))=5$ by constructing a total 5-colouring, while for others, we will demonstrate $\chi''(C_{n}(1,3))=6$ by excluding $\chi''(C_{n}(1,3))$=5.

With these preliminaries, we may now start to determine  $\chi''(C_{n}(1,3))$ for all $n\geq 7$. Our paper is organized as follows. We will first determine the chromatic numbers of $C_{n}(1, 3)$ for $n=5p+9q$ with $p$ and $q$ being  nonnegative integers in section 3, and then we determine the chromatic numbers of $C_{n}(1, 3)$ for $n=11,16,21,26,31$ in Section 4, for $n=7,12,17,22$ in Section 5, and for $n=8,13$ in Section 6.  Section \ref{Sec4} is the conclusion.

\section{Total colouring of  $C_{n}(1, 3)$ for $n = 5p + 9q$ }\label{Secpq}

 In this section, we  study the total colouring of  $C_{n}(1, 3)$ for $n = 5p + 9q$.

\begin{lemma}
$\chi''(C_{n}(1,3)) = 5$ for $n = 5p + 9q$
with nonnegative integers $p,q$.
\label{Lempq}
\end{lemma}

\begin{proof}
For simplicity, we write $(i_1 i_2 \ldots i_t)^p$ for the sequence obtained by repeating $i_1 i_2 \ldots i_t$ exactly $p$ times, where  $i_j \in \{1,2,3,4,5\}$. For example, $(12345)^2 = 1234512345$.
Let	
		$V$ = $\left\{v_i : 0\le i \le n-1\right\}$,	
	$E_1$ = $\left\{v_iv_{i+1} : 0\le i \le n-1\right\}$,
	$E_2$ = $\left\{v_iv_{i+3} : 0\le i \le n-1\right\}$
 and $\sigma(C_{n}(1, 3))=(\sigma(V),\sigma(E_1), \sigma(E_2))$.

We construct  $\sigma(C_{n}(1, 3))$ for $n = 5p + 9q$ as follows:\\
\begin{small}
	\noindent$\sigma(C_n(1, 3))=((24351)^p(212534121)^q, (12123)^p(453453453)^q, (45534)^p(121212534)^q)$.
\end{small}

The above construction $\sigma(C_n(1,3))$ assigns to each vertex and edge of $C_n(1,3)$ a colour from $\{1,2,3,4,5\}$ such that
adjacent vertices receive distinct colours,  adjacent edges receive distinct colours, 
  and every vertex receives a colour different from those of its incident edges.
So, it is a total 5-colouring of $C_{5p + 9q}(1,3)$, from which it follows that $\chi''(C_n(1,3)) \leq 5$ for $n = 5p + 9q$. 
 It is also known that
 $\chi''(C_{n}(1,3)) \geq 5$.  Hence, $\chi''(C_{n}(1,3)) = 5$ for $n = 5p + 9q$. 
\end{proof}	
Figure \ref{Figc9-10-14} shows $\sigma(C_{9}(1, 3))$, $\sigma(C_{10}(1, 3))$and $\sigma(C_{14}(1, 3))$.
\begin{figure}[h]
	\centering
	\includegraphics[scale=0.56]{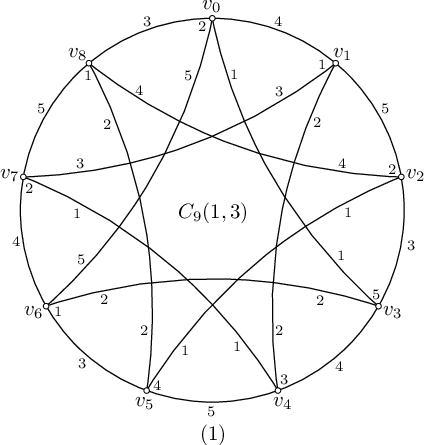}~~~~
    \includegraphics[scale=0.56]{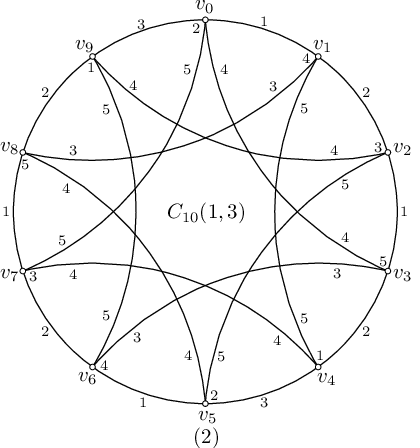}~~~~
   \includegraphics[scale=0.56]{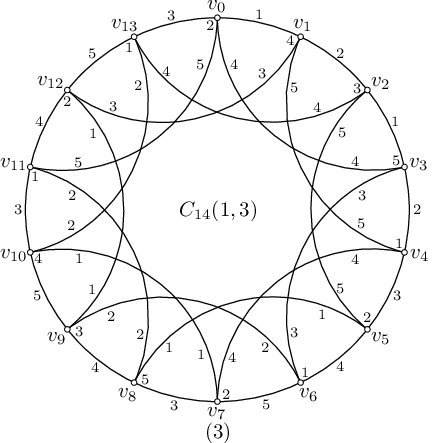}   
	\caption{$\sigma(C_{9}(1, 3))$, $\sigma(C_{10}(1, 3))$and $\sigma(C_{14}(1, 3))$.}
	\label{Figc9-10-14}
\end{figure}

\section{Total colouring of  $C_{n}(1, 3)$ for $n=11,16,21,26,31$ }\label{Sec51}

In this section, we  study the total colouring of  $C_{n}(1, 3)$ for $n=11,16,21,26,31$.

\begin{lemma}
$\chi''(C_{n}(1,3)) = 5$ for  $n=11,16,21,26,31$.
\label{Lem51}
\end{lemma}	

\begin{proof}
We construct  $\sigma(C_{n}(1, 3))$ for $n$ = 11,16,21,26,31  as follows:\\
 \begin{small}
		\noindent$\sigma(C_{11}(1, 3))=(25354543431, 12121212123, 43435354545)$,\\
		\noindent$\sigma(C_{16}(1, 3))=(2453534242353524, 1212121314141415, 4345453525232353)$,\\
		\noindent$\sigma(C_{21}(1, 3))=(234345453512345123451, 121212121231213451323, 453534345454532214545)$,\\		
		\noindent$\sigma(C_{26}(1, 3))=(24535343454535141252313421,12121212121212323434545145,\\
		\indent~~~~~~~~~~~~~~~~~43454535343454515121232353)$,\\
		\noindent$\sigma(C_{31}(1, 3))=(2343454535123451234512345123451,1212121212312134513231213451323,\\
		\indent~~~~~~~~~~~~~~~~~4535343454545322145454532214545)$.
	\end{small}

These constructions use only colours from $\{1,2,3,4,5\}$, and in each case,
adjacent vertices receive distinct colours, adjacent edges receive distinct colours, and
no vertex shares a colour with any of its incident edges.
Hence, each gives a total 5-colouring of $C_n(1,3)$, which yields $\chi''(C_n(1,3)) \leq 5$ for $n = 11, 16, 21, 26, 31$.  
Moreover, $\chi''(C_{n}(1,3)) \geq 5$. Therefore, $\chi''(C_{n}(1,3)) = 5$ for $n = 11,16,21,26,31$.
\end{proof}
Figure \ref{Figc11-16} shows $\sigma(C_{11}(1,3))$ and $\sigma(C_{16}(1,3))$.
\begin{figure}[h]
	\centering
	\includegraphics[scale=0.62]{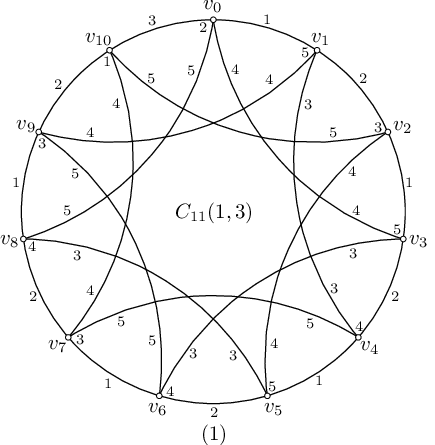}~~~~~~~
    \includegraphics[scale=0.62]{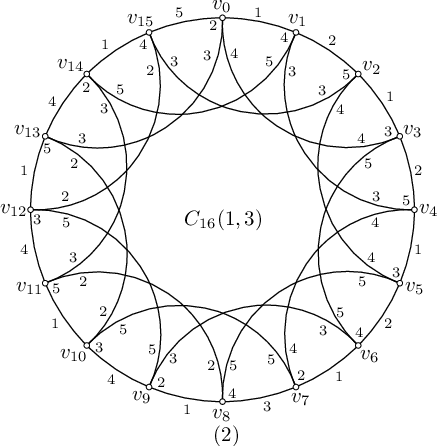}
	\caption{$\sigma(C_{11}(1,3))$ and $\sigma(C_{16}(1,3))$.}
	\label{Figc11-16}
\end{figure}

\section{Total colouring of  $C_{n}(1, 3)$ for $n=7,12,17,22$ }\label{Sec52}

In this section, we  study the total colouring of  $C_{n}(1, 3)$ for $n=7,12,17,22$.

\subsection{Total colouring of  $C_{7}(1, 3)$  }

\begin{lemma}
$\chi''(C_{7}(1,3)) = 6$.
\label{Lem7}
\end{lemma}
	
\begin{proof}
Suppose to the contrary that   $C_7(1,3)$ has a total $5$-colouring. Since the maximum independent set in $C_7(1,3)$ has size $2$, then there are  $|V_{j}|=1$ $(1\leq j\leq5)$ by Lemma \ref{Lem1},  a contradiction to $\sum_{1\leq j\leq5}|V_{j}|=7$.  Thus, $C_7(1,3)$ does not have a total 5-colouring. Instead, $\chi''(C_{7}(1,3)) \geq 6$.
On the other hand, according to Lemma \ref{Lem1T}, $\chi''(C_{7}(1,3)) \leq 6$. There must be
  $\chi''(C_{7}(1,3)) = 6$.
   \end{proof}

\subsection{Total colouring of  $C_{12}(1, 3)$  }
The total chromatic number of $C_{12}(1, 3)$ has been obtained in Ref. \cite{Far23,Nig21}. For completeness, we restate it here.
\begin{lemma}
$\chi''(C_{12}(1,3)) = 6$.
\label{Lem12}
\end{lemma}

\subsection{Total colouring of  $C_{17}(1, 3)$  }
\begin{lemma}
$\chi''(C_{17}(1,3)) = 6$.
\label{Lem17}
\end{lemma}

\begin{proof}
Suppose to the contrary that  $C_{17}(1,3)$ has a total $5$-colouring.  The vertices   coloured $j$ is explicitly denoted as $V_j=\{v_{j_{0}},v_{j_{1}},\cdots\,v_{j_{(t-1)}}\}$, $1\leq j\leq5$. Without  loss of generality, we assume $|V_{j_{1}}|\geq|V_{j_{2}}|$  for $1\leq j_{1}\leq j_{2}\leq5$. We further define  $d_{js}$=($n+j_{s+1}-j_s$)~(mod~$n$) (indices of $j$ are  modulo $t$), $0\leq s\leq t-1$. Then, there is $d_{js}\in\{2,4,5,6,7,8,9,10,11,12,13,15\}$ for $C_{17}(1,3)$. Since the maximun independent set in $C_{17}(1,3)$ has size $7$, then $|V_{j}|\in \{7,5,3,1\}(1\leq j\leq5)$  by the Lemma \ref{Lem1}. From $\Sigma_{1\leq j\leq5}|V_j|=17$, we have $|V_1|=7$ or $|V_1|=5$.

If $|V_1| = 7$, then $|\{d_{1s}\mid d_{1s}=2\}|\leq 6$. However, when $|\{d_{1s}\mid d_{1s}=2\}|\leq 4$, there is $\sum_{1\leq s\leq7}d_{1s}\geq 2\times 4+4\times 3=20>17$, a contradiction to $\sum_{1\leq s\leq7}d_{1s}=17$. When $|\{d_{1s}|d_{1s}=2\}|\geq5$, there must be an integer $s$ such that $(d_{1s},d_{1(s+1)},d_{1(s+2)})=(2,2,2)$. Without  loss of generality, let $v_{0},v_{2},v_{4},v_{6}\in V_{1}$. It follows  $\sigma(v_{3}), \sigma(v_{3}v_{2}),\sigma(v_{3}v_{4}),\sigma(v_{3}v_{0}),\sigma(v_{3}v_{6})$ $\neq 1$  (see Figure \ref{Fig24}(1)), a contradiction to the assumption that $C_{17}(1, 3)$  has a total 5-colouring.

If $|V_1|= 5$, then by symmetry, we need only consider $(d_{11},d_{12},d_{13})\in\{(2,2,4),(2,2,5),$ $(2,4,2),(2,4,4)\}$, since  $d_{1s}\in\{2,4,5,6,7,8,9,10,11,12,13,15\}$,  $\sum_{1\leq s\leq5}d_{1s}=17$,
 and $(d_{11},d_{12},d_{13})\neq (2,2,2)$ as demonstrated above.

Case 1. $(d_{11},d_{12},d_{13})=(2,2,4)$. Without loss of generality, we let $V_1=\{v_0,v_2,v_4,v_8\}$. Then $\sigma(v_{3}),\sigma(v_{3}v_{2}), \sigma(v_{3}v_{4})$, $\sigma(v_{3}v_{0})\neq 1$, $\sigma(v_{3}v_{6})=1$. It follows  $\sigma(v_{5}),\sigma(v_{5}v_{4}), \sigma(v_{5}v_{6}),$ $\sigma(v_{5}v_{2}),\sigma(v_{5}v_{8})\neq 1$ (see Figure \ref{Fig24}(2)), a contradiction  to the assumption.

Case 2.  $(d_{11},d_{12},d_{13})=(2,2,5)$. We  need only consider $(d_{11},d_{12},d_{13},d_{14},d_{15})=(2,2,5,2,6)$, since $(d_{11},d_{12},d_{13},d_{14},d_{15})=(2,2,5,6,2)$ has been embodied in $(d_{1s},d_{1(s+1)},$ $d_{1(s+2)})= (2,2,2)$. Let $V_1=\{v_0,v_2,v_4,v_9,v_{11}\}$. Then $\sigma(v_{1}),\sigma(v_{1}v_{0}),\sigma(v_{1}v_{2}),\sigma(v_{1}v_{4})\neq 1$, $\sigma(v_{1}v_{15})=1$. It follows $\sigma(v_{12}v_{13})=1$. We have
 $\sigma(v_{14}),\sigma(v_{14}v_{15}), \sigma(v_{14}v_{13})$, $\sigma(v_{14}v_{0}),$ $\sigma(v_{14}v_{11})\neq 1$ (see Figure \ref{Fig24}(3)), a contradiction  to the assumption.

Case 3. $(d_{11},d_{12},d_{13})=(2,4,2)$. Let $V_1=\{v_0,v_2,v_6,v_8\}$.
Then  $\sigma(v_{3}),\sigma(v_{3}v_{2}),\sigma(v_{3}v_{0}),$ $\sigma(v_{3}v_{6})\neq 1$, $\sigma(v_{3}v_{4})=1$. It follows $\sigma(v_{5}),\sigma(v_{5}v_{4}),\sigma(v_{5}v_{6}), \sigma(v_{5}v_{2}),\sigma(v_{5}v_{8})\neq 1$ (see Figure \ref{Fig24}(4)), a contradiction to the assumption.

Case 4. $(d_{11},d_{12},d_{13})=(2,4,4)$. Let $V_1=\{v_0,v_2,v_6,v_{10}\}$. Then $\sigma(v_{3}v_{4})=1$. It follows $\sigma(v_{5}v_{8})=1$. We have $\sigma(v_{7}),\sigma(v_{7}v_{6}),\sigma(v_{7}v_{8}),\sigma(v_{7}v_{4}),\sigma(v_{7}v_{10})\neq 1$  (see Figure \ref{Fig24}(5)), a contradiction  to the assumption.
\begin{figure}[h]
\centering
	\includegraphics[scale=0.68]{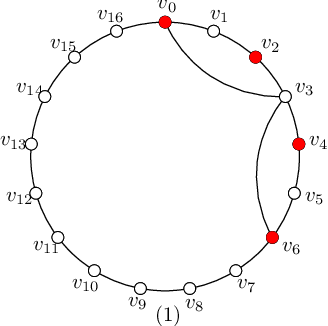}~~~~~
   \includegraphics[scale=0.68]{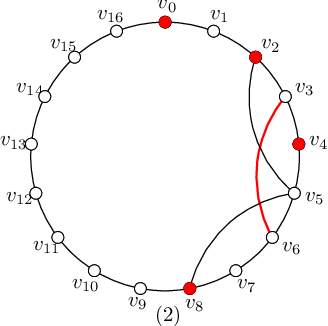}~~~~~
    \includegraphics[scale=0.68]{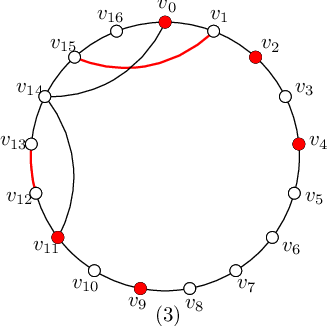}

         \vspace{5pt}
  \includegraphics[scale=0.68]{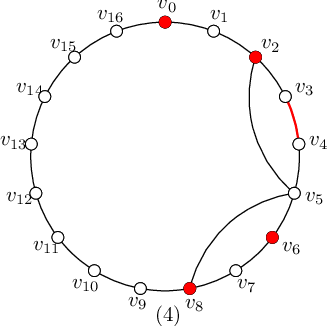}~~~~~~~~~~~
  \includegraphics[scale=0.68]{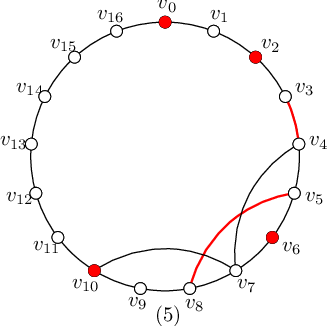}
\caption{$\sigma(C_{17}(1, 3))$ for$|V_1| =7,5$.}
\label{Fig24}
\end{figure}

 Thus, $C_{17}(1,3)$ does not have a total 5-colouring, i.e. $\chi''(C_{17}(1,3)) \geq 6$.
  However, according to Lemma \ref{Lem1T}, $\chi''(C_{17}(1,3)) \leq 6$. So,
 $\chi''(C_{17}(1,3)) = 6$.
\end{proof}

\subsection{Total colouring of  $C_{22}(1, 3)$  }

\begin{lemma}
$\chi''(C_{22}(1,3)) = 5$.
\label{Lem22}
\end{lemma}

\begin{proof}
We construct  $\sigma(C_{22}(1, 3))$  as follows:

 \begin{small}
      \noindent$\sigma(C_{22}(1, 3))=(2545353434545353124341, 1212121212121212312124,
      3434545353434545453535)$.
\end{small}

In this construction,  
every vertex and edge of $C_{22}(1,3)$ receives a colour from $\{1,2,3,4,5\}$, with adjacent vertices receiving distinct colours, adjacent edges receiving distinct colours, and no vertex sharing a colour with any of its incident edges. 
Hence, it gives a total 5-colouring of $C_{22}(1,3)$, implying $\chi''(C_{22}(1,3)) \leq 5$.  
On the other hand, $\chi''(C_{22}(1,3)) \geq 5$. Therefore, $\chi''(C_{22}(1,3)) = 5$.
\end{proof}
Figure \ref{Figc22} shows $\sigma(C_{22}(1,3))$.
\begin{figure}[h]
	\centering
	\includegraphics[scale=0.71]{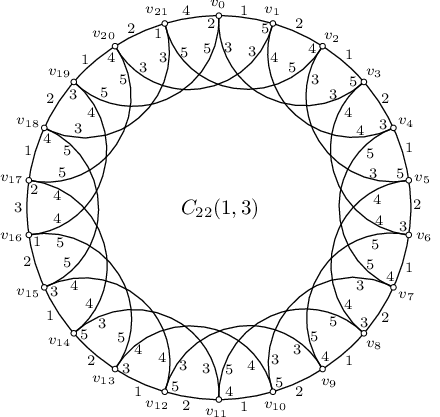}
    	\caption{$\sigma(C_{22}(1,3))$. }
	\label{Figc22}
\end{figure}

\section{Total colouring of  $C_{n}(1, 3)$ for $n=8,13$ }\label{Sec53}

In this section, we  study the total colouring of  $C_{n}(1, 3)$ for $n=8,13$.

\subsection{Total colouring of  $C_{8}(1, 3)$ }

\begin{lemma}
$\chi''(C_{8}(1,3)) = 6$.
\label{Lem8}
\end{lemma}
Since $C_{8}(1,3)$ is isomorphic to $K_{4,4}$, of which the total chromatic number is $6$ \cite{Beh67}, the total chromatic number of $C_{8}(1,3)$ is $6$ too.

\subsection{Total colouring of  $C_{13}(1, 3)$ }
We use the same expressions of $V_{j}$ and $d_{js}$ as  those in the proof of Lemma \ref{Lem17}. There is $d_{js}\in \{2,4,5,6,7,8,9,11\}$ for $C_{13}(1,3)$.

\begin{lemma}
If~ $C_{13}(1, 3)$ has a total $5$-colouring,
 then $|V_{1}|=|V_{2}|=|V_{3}|=|V_{4}|= 3$ and $|V_{5}|=1$.
\label{Lem71}
\end{lemma}

\begin{proof}
 If~ $C_{13}(1, 3)$ has a total $5$-colouring,  then by Lemma \ref{Lem1}, $|V_{j}|\in \{5,3,1\}(1\leq j\leq5) $ since the maximum independent set in $C_{13}(1,3)$ has size $5$.

We first consider the case of $|V_{1}| =5$. Since  $d_{1s}\in \{2,4,5,6,7,8,9,11\}$ and $\Sigma_{1\leq s\leq5}d_{1s}=13$, we have $|\{d_{1s} \mid d_{1s}=2\}|=4$.
         Without  loss of generality, let $V_{1}=\{v_{0},v_2,v_4,v_6,v_8\}$. 
        We then have  $\sigma(v_{3}),\sigma(v_{3}v_{2})$, $\sigma(v_{3}v_{4}),\sigma(v_{3}v_{0}),\sigma(v_{3}v_{6})\neq 1$ (see Figure \ref{Fig13-1}), a contradiction  to the precondition that $C_{13}(1,3)$ has a total $5$-colouring.
\begin{figure}[h]
 \centering
   \includegraphics[scale=0.68]{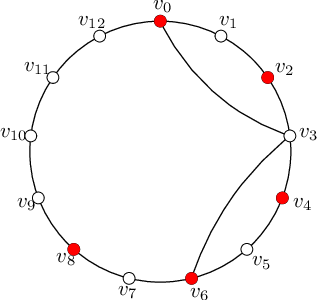}
  \caption{$\sigma(C_{13}(1, 3))$ for $|V_1| =5$.}
	 \label{Fig13-1}
\end{figure}

Hence, $|V_1|= 3$.  From $\Sigma_{1\leq j\leq5}|V_j|=13$, we have
$|V_1|= |V_2|=|V_3|=|V_4|=3$ and $|V_5|=1$. 
\end{proof}

\begin{lemma}
If~ $C_{13}(1, 3)$ has a total $5$-colouring, then
 $(\{d_{j1},d_{j2},d_{j3}\},c_{j})$ $\in$ $\{(\{2,4,7\},4),$ $(\{2,5,6\},3),(\{2,5,6\},$ $4),(\{4,4,5\},3)\}$  and there are  at least three colours $j$ such that    $(\{d_{j1},d_{j2},d_{j3}\},c_{j})\in\{(\{2,5,6\},3),(\{4,4,5\},3)\}$ for $1\leq j\leq 4$, where $c_{j}$ is the number of chords  coloured $j$.
 \label{Lem72}
\end{lemma}

\begin{proof}
If~ $C_{13}(1, 3)$ has a total $5$-colouring, then by Lemma \ref{Lem71},
 $|V_j|= 3$ for $1\leq j\leq 4$.  Since  $d_{js}\in \{2,4,5,6,7,8,9,11\}$ and  $\sum_{1\leq s\leq3}d_{js}=13$, we have $\{d_{j1},d_{j2},d_{j3}\}\in\{\{2,2,9\},\{2,4,7\},\{2,5,6\},\{4,4,5\}\}$.

	Case 1. $\{d_{j1},d_{j2},d_{j3}\}=\{2,2,9\}$.
	We may let $V_j=\{v_{0},v_2,v_4\}$.
Then $\sigma(v_{1}),\sigma(v_{1}v_{0}),$ $\sigma(v_{1}v_{2}), \sigma(v_{1}v_{4})\neq j$, $\sigma(v_{1}v_{11})=j$. It follows $\sigma(v_{12}v_{9})=j$, $\sigma(v_{10}v_{7})=j$, $\sigma(v_{8}v_{5})=j$, $\sigma(v_{6}v_{3})=j$ and $c_{j}=5$ (see  Figure \ref{Figd229-247-256-445}(1)).

    Case 2. $\{d_{j1},d_{j2},d_{j3}\}=\{2,4,7\}$.
Let $V_j=\{v_{0},v_2,v_6\}$.
Then $\sigma(v_{3}),\sigma(v_{3}v_2),\sigma(v_{3}v_0),$ $\sigma(v_3v_6) \neq  j$, $\sigma(v_{3}v_{4})=j$. It follows $\sigma(v_{1}v_{11})=j$, $\sigma(v_{12}v_{9})=j$, $\sigma(v_{10}v_{7})=j$, $\sigma(v_{8}v_{5})=j$ and $c_{j}=4$  (see  Figure \ref{Figd229-247-256-445}(2)).	

    Case 3. $\{d_{j1},d_{j2},d_{j3}\}=\{2,5,6\}$. Let  $V_j=\{v_{0},v_2,v_7\}$.
	Then $\sigma(v_1),\sigma(v_1v_0),\sigma(v_1v_2) \neq  j$,  $\sigma(v_{1}v_{11})=j$ or $\sigma(v_1v_4) = j$.
	If $\sigma(v_{1}v_{11}) = j$, then  $\sigma(v_{12}v_{9}) = j$. It follows $\sigma(v_{10}),$ $\sigma(v_{10}v_{11}),\sigma(v_{10}v_{9}),\sigma(v_{10}v_{0}), \sigma(v_{10}v_{7})\neq  j$ (see Figure \ref{Figd229-247-256-445}(3)),  a contradiction  to the precondition. 	If $\sigma(v_1v_4) = j$, then
$\sigma(v_3v_6)=j$, $\sigma(v_5v_8) = j$. It follows
 $\sigma(v_9v_{10}) = j$,  $\sigma(v_{11}v_{12}) = j$ and $c_{j}=3$   (see Figure \ref{Figd229-247-256-445}(4))
	or $\sigma(v_9v_{12}) = j$,  $\sigma(v_{10}v_{11}) = j$ and $c_{j}=4$ (see Figure \ref{Figd229-247-256-445}(5)).	

    Case 4. $\{d_{j1},d_{j2},d_{j3}\}=\{4,4,5\}$. Let $V_j=\{v_{0},v_4,v_8\}$.
Then $\sigma(v_3),\sigma(v_3v_0),\sigma(v_3v_4) \neq  j$, $\sigma(v_{3}v_{2})=j$
or $\sigma(v_{3}v_{6})=j$.  If  $\sigma(v_3v_2) = j$, then  $\sigma(v_1v_{11}) = j$, $\sigma(v_{12}v_{9}) = j$,   $\sigma(v_{10}v_{7}) = j$, $\sigma(v_6v_{5}) = j$ and  $c_{j}=3$ (see Figure \ref{Figd229-247-256-445}(6)).
  If $\sigma(v_3v_6) = j$, then  $\sigma(v_5v_2) = j$, $\sigma(v_1v_{11}) = j$, $\sigma(v_{12}v_{9}) = j$,   $\sigma(v_{10}v_{7}) = j$ and $c_{j}=5$	(see Figure \ref{Figd229-247-256-445}(7)).

By Cases 1--4, we have $c_{j}\in\{3,4,5\}$ for $1\leq j\leq4$.
From $\sum_{1\leq j\leq4}c_{j}\leq13$, there is $c_{j}\leq4$,    $(\{d_{j1},d_{j2},d_{j3}\},c_{j})\in\{(\{2,4,7\},4),(\{2,5,6\},3)$, $(\{2,5,6\},4),(\{4,4,5\},3)\}$  and there are  at least three colours $j$ such that    $(\{d_{j1},d_{j2},d_{j3}\},c_{j})\in\{(\{2,5,6\},3),(\{4,4,5\},3)\}$ for $1\leq j\leq 4$.
\end{proof}
\begin{figure}[h]
   \centering
   \includegraphics[scale=0.69]{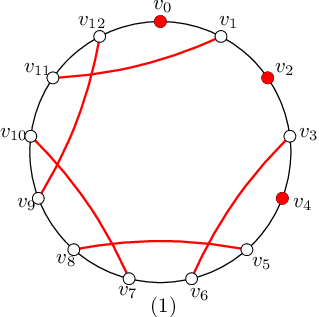}~
   \includegraphics[scale=0.69]{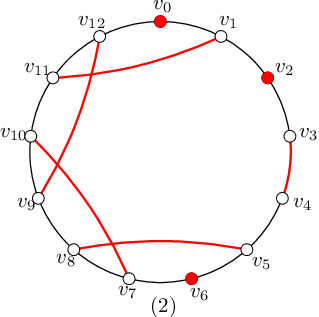}~
	\includegraphics[scale=0.69]{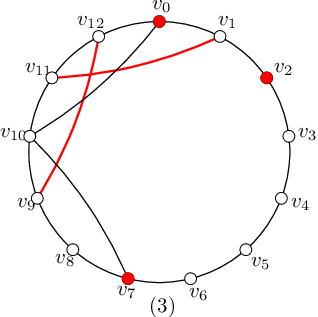}~
     \includegraphics[scale=0.69]{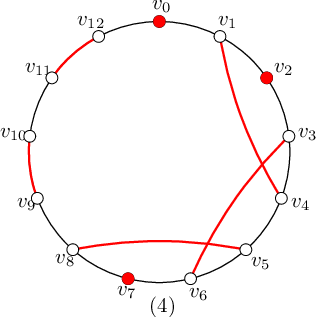}
      
         \vspace{7pt}
	~~~~~\includegraphics[scale=0.69]{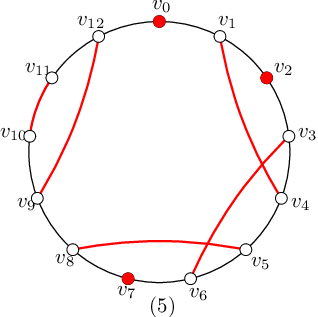}~~~~	
     \includegraphics[scale=0.69]{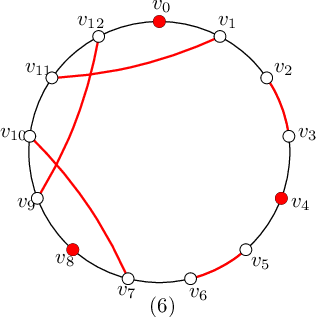}~~~~
	\includegraphics[scale=0.69]{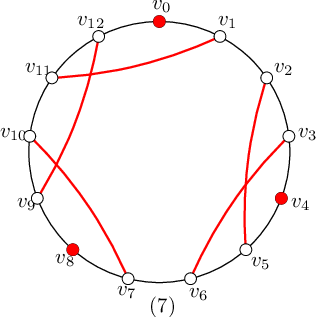}	   
	\caption{$\sigma(C_{13}(1, 3))$ for $|V_1| =3$.}
	 \label{Figd229-247-256-445}
\end{figure}
Further, according to Cases 2--4, we have  the following lemma.

\begin{lemma}
If~ $C_{13}(1, 3)$ has a total $5$-colouring, then
  $\sigma(v_{j_{s}-1}v_{j_{s}+2}) = \sigma(v_{j_{s}+1}v_{j_{s}+4})$ = $\sigma(v_{j_{s}+3}v_{j_{s}+6}) = \sigma(v_{j_{s}})$ for $d_{js}=5$
  and  $\sigma(v_{j_{s}-1}v_{j_{s}+2})$ = $\sigma(v_{j_{s}+1}v_{j_{s}+4})$ = $\sigma(v_{j_{s}+3}v_{j_{s}+6})$ = $\sigma(v_{j_{s}+5}v_{j_{s}+8})$ = $\sigma(v_{j_{s}})$ for $d_{js}=7$, where $1\leq j\leq4$ and indices of $v$ are  modulo 13.
\label{Lem773}
\end{lemma}

For $1\leq j_{1},j_{2}\leq 4$, let $d^{2}_{j_{1},j_{2}}=i_{1}-i_{2}$  where $\sigma(v_{i_{1}})=\sigma(v_{i_{1}+5})=j_{1}$ and $\sigma(v_{i_{2}})=\sigma(v_{i_{2}+5})=j_{2}$.
  Then $d^{2}_{j_{1},j_{2}}\in\{1,2,3,4,6,7,9,10,11,12\}$.

\begin{lemma}
If~ $C_{13}(1, 3)$ has a total $5$-colouring,    then
$d^{2}_{j_{1},j_{2}}\in\{1,3,6,7,10,12\}$. 
\label{Lem73}
\end{lemma}

\begin{proof}
   Without loss of generality, let
   $\sigma(v_{0}) = \sigma(v_5) =j_{1}$.
 By Lemma \ref{Lem773},  $\sigma(v_3v_{6}) = j_{1}$.
 If $d^{2}_{j_{1},j_{2}}\in\{2,4\}$, then $\sigma(v_3v_{6}) = j_{2}$  (see Figure \ref{Figp224}), a contradiction.  So, $d^{2}_{j_{1},j_{2}}\notin\{2,4\}$. By symmetry, $d^{2}_{j_{1},j_{2}}\notin\{11,9\}$.	Hence,	$d^{2}_{j_{1},j_{2}}\in\{1,3,6,7,10,12\}$.
\begin{figure}[h]
	\centering
    \includegraphics[scale=0.69]{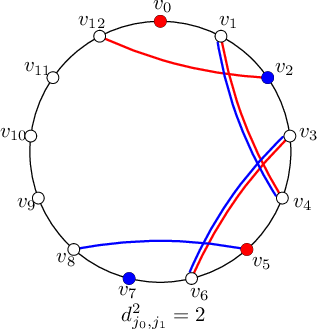}~~~~~~~~~~~~~
	\includegraphics[scale=0.69]{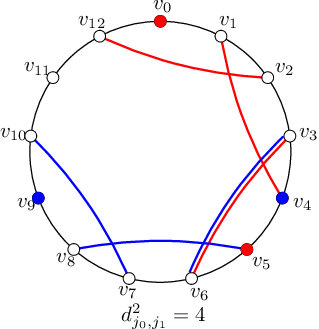}
   	\caption{$\sigma(C_{13}(1, 3))$ for $d^{2}_{j_{1},j_{2}}=2,4$.}
	\label{Figp224}
\end{figure}
\end{proof}
\begin{lemma}

$\chi''(C_{13}(1,3)) = 6$.

\label{Lem13}
\end{lemma}

\begin{proof}	
Suppose  to the contrary that  $C_{13}(1,3)$ has a total $5$-colouring.  By Lemmas \ref{Lem71}--\ref{Lem72}, 
  $|V_j|= 3$, $(\{d_{j1},d_{j2},d_{j3}\}$, $c_{j})$ $\in\{(\{2,4,7\},4)$, $(\{2,5,6\},3)$, $(\{2,5,6\},4),(\{4,4,5\},3)\}$,
   and there are  at least three colours $j$ such that    $(\{d_{j1},d_{j2},d_{j3}\},c_{j})\in\{(\{2,5,6\},3),(\{4,4,5\},3)\}$ for $1\leq j\leq 4$.
   Without loss of generality, let  $(\{d_{j1},d_{j2},d_{j3}\},c_{j})\in\{(\{2,5,6\},3),(\{4,4,5\},3)\}$
   for $j=1,2,3$,
    and let $\sigma(v_{0}) = \sigma(v_5) =1$. Then   $\sigma(v_{12}v_{2}) = \sigma(v_{1}v_{4}) = \sigma(v_{3}v_{6}) = 1$. By symmetry, we need only consider $d^{2}_{1,2}\in\{1,3,6\}$.

Case 1. $d^{2}_{1,2}=1$. Then $\sigma(v_{0}v_{3}) = \sigma(v_{2}v_{5}) = \sigma(v_{4}v_{7}) = 2$ and $d^{2}_{1,3}\in\{3,6,7,10,12\}$. If $d^{2}_{1,3}\in\{3,6,10,12\}$, then  $d^{2}_{2,3}\notin\{1,3,6,7,10,12\}$. So,  $d^{2}_{1,3}=7$. It follows  $\sigma(v_{6}v_{9}) = \sigma(v_{8}v_{11}) = \sigma(v_{10}v_{0}) = 3$ and $d^{2}_{1,4}\in\{3,6,10,12\}$. If $d^{2}_{1,4}\in\{3,6,10,12\}$, then $d^{2}_{2,4}\notin\{1,3,6,7,10,12\}$. Hence,  $(\{d_{41},d_{42},d_{43}\},c_{4})$ =$(\{2,4,7\},4)$ and
 $V_{4}\in\{\{v_{2},v_{9},v_{11}\},\{v_{4},v_{11},$ $v_{2}\},\{v_{8},v_{2},v_{4}\},\{v_{10},v_{4},v_{8}\}\}$.

If $V_{4}\in \{\{v_{2},v_{9},v_{11}\},\{v_{4},v_{11},v_{2}\}\}$, then  $\sigma(v_{3}v_{6}) =4$ (see Figure \ref{Fign21} (1)--(2)), a contradiction.

  If $V_{4}\in\{\{v_{8},v_{2},v_{4}\},\{v_{10},v_{4},v_{8}\}\}$, then  $\sigma(v_{0}v_{3}) = 4$ (see Figure \ref{Fign21} (3)--(4)), a contradiction.
	\begin{figure}[h]
	\centering
	 \includegraphics[scale=0.69]{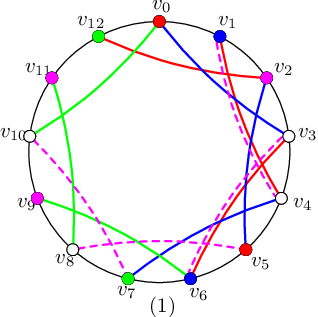}~
     \includegraphics[scale=0.69]{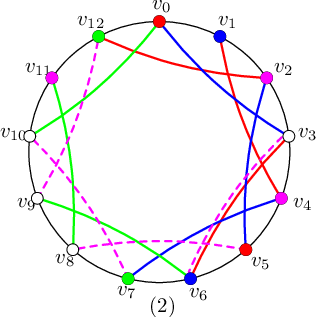}~
     \includegraphics[scale=0.69]{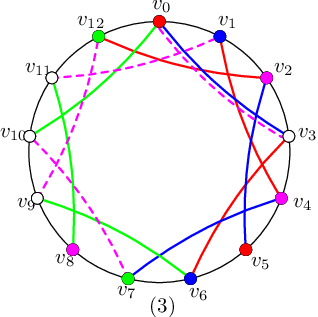}~
     \includegraphics[scale=0.69]{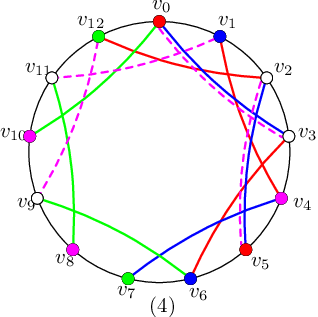}
    \caption{ $\sigma(C_{13}(1, 3))$ for  $d^{2}_{1,2} =1$.}
	\label{Fign21}
\end{figure}

Case 2. $d^{2}_{1,2}=3$. Then $\sigma(v_{2}v_{5}) = \sigma(v_{4}v_{7}) = \sigma(v_{6}v_{9}) = 2$ and $d^{2}_{1,3}\in\{1,6,7,10,12\}$. If $d^{2}_{1,3}\in\{1,7,12\}$, then $d^{2}_{2,3}\notin\{1,3,6,7,10,12\}$. Hence,  $d^{2}_{1,3}\in\{6,10\}$.

\indent Case 2.1. $d^{2}_{1,3}=6$. Then $\sigma(v_{5}v_{8}) = \sigma(v_{7}v_{10}) = \sigma(v_{9}v_{12}) = 3$ and  $d^{2}_{1,4}\in\{1,7,10,12\}$.	
If $d^{2}_{1,4}\in\{1,7,12\}$, then $d^{2}_{2,4}\notin\{1,3,6,7,10,12\}$. If  $d^{2}_{1,4}=10$, then $d^{2}_{3,4}\notin\{1,3,6,7,10,12\}$. So,   $(\{d_{41},d_{42},d_{43}\},c_{4})=(\{2,4,7\},4)$. It forces  $\sigma(v_{8}v_{11})$ =$\sigma(v_{11}v_{1})=\sigma(v_{10}v_{0})=\sigma(v_{0}v_{3})$ $=4$  (see Figure \ref{Fign22} (1)), a contradiction to the requirements of total colouring.

Case 2.2.  $d^{2}_{1,3}=10$. Then $\sigma(v_{9}v_{12}) = \sigma(v_{11}v_{1}) = \sigma(v_{0}v_{3}) = 3$ and $d^{2}_{1,4}\in\{1,6,7,12\}$.	
If $d^{2}_{1,4}\in\{1,7,12\}$, then $d^{2}_{2,4}\notin\{1,3,6,7,10,12\}$. If  $d^{2}_{1,4}=6$, then $d^{2}_{3,4}\notin\{1,3,6,7,10,12\}$. So   $(\{d_{41},d_{42},d_{43}\},c_{4})=(\{2,4,7\},4)$. It forces  $\sigma(v_{5}v_{8})=\sigma(v_{8}v_{11})=\sigma(v_{7}v_{10})=\sigma(v_{10}v_{0})=4$  (see Figure \ref{Fign22} (2)), a contradiction to the requirements of total colouring.
\begin{figure}[h]
	\centering
	 \includegraphics[scale=0.69]{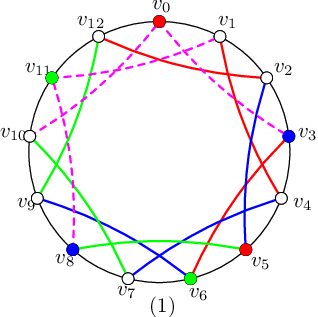}~~~~~~~~~~~~~
     \includegraphics[scale=0.69]{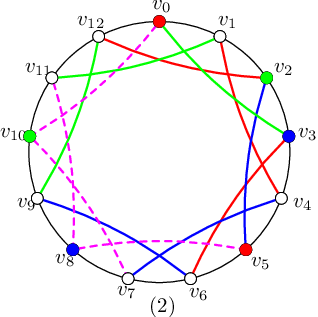}
     \caption{ $\sigma(C_{13}(1, 3))$ for  $d^{2}_{1,2}=3$.}
	\label{Fign22}
\end{figure}

Case 3. $d^{2}_{1,2}=6$. Then  $\sigma(v_{5}v_{8}) = \sigma(v_{7}v_{10}) = \sigma(v_{9}v_{12}) = 2$ and $d^{2}_{1,3}\in\{1,3,7,10,12\}$. If $d^{2}_{1,3}\in\{1,10\}$, then $d^{2}_{2,3}\notin\{1,3,6,7,10,12\}$. Hence,  $d^{2}_{1,3}\in\{3,7,12\}$. 

Case 3.1.  $d^{2}_{1,3}=3$.  This case is  analogous to  Case 2.1.

Case 3.2.  $d^{2}_{1,3}=7$. Then  $\sigma(v_{6}v_{9}) = \sigma(v_{8}v_{11}) = \sigma(v_{10}v_{0}) = 3$ and $d^{2}_{1,4}\in\{1,3,10,12\}$.	
If $d^{2}_{1,4}\in\{1,10\}$, then $d^{2}_{2,4}\notin\{1,3,6,7,10,12\}$. If  $d^{2}_{1,4}\in\{3,12\}$, then $d^{2}_{3,4}\notin\{1,3,6,7,10,12\}$. So,   $(\{d_{41},d_{42},d_{43}\},c_{4})=(\{2,4,7\},4)$ and
$V_{4}\in\{\{v_{1},v_{8},v_{10}\},\{v_{3},v_{10},v_{1}\},\{v_{8},v_{2},v_{4}\},\{v_{10},$ $v_{4},v_{8}\}\}$.

If $V_{4}\in \{\{v_{1},v_{8},v_{10}\},\{v_{3},v_{10},v_{1}\}\}$, then $\sigma(v_{6}v_{9}) = 4$ (see Figure \ref{Fign24}(1)--(2)), a contradiction.

If  $V_{4}\in\{\{v_{8},v_{2},v_{4}\},\{v_{10},v_{4},v_{8}\}\}$, then   $\sigma(v_{9}v_{12}) = 4$   (see Figure \ref{Fign24}(3)--(4)),  a contradiction.

Case 3.3.  $d^{2}_{1,3}=12$.  Then $\sigma(v_{11}v_{1}) = \sigma(v_{0}v_{3}) = \sigma(v_{2}v_{5}) = 3$ and $d^{2}_{1,4}\in\{1,3,7,10\}$.	
If $d^{2}_{1,4}\in\{1,10\}$, then $d^{2}_{2,4}\notin\{1,3,6,7,10,12\}$. If  $d^{2}_{1,4}\in\{3,7\}$, then $d^{2}_{3,4}\notin\{1,3,6,7,10,12\}$. So   $(\{d_{41},d_{42},d_{43}\},c_{4})=(\{2,4,7\},4)$ and
$V_{4}\in\{\{v_{1},v_{8},v_{10}\}$, $\{v_{3},v_{10},v_{1}\},\{v_{7},v_{1},v_{3}\},\{v_{9},$ $v_{3},v_{7}\}\}$.

If $V_{4}\in \{\{v_{1},v_{8},v_{10}\},\{v_{3},v_{10},v_{1}\}\}$, then  $\sigma(v_{2}v_{5}) = 4$ (see Figure \ref{Fign24}(5)--(6)), a contradiction.

If $V_{4}\in\{\{v_{7},v_{1},v_{3}\},\{v_{9},v_{3},v_{7}\}\}$, then  $\sigma(v_{12}v_{2}) = 4$ (see Figure \ref{Fign24}(7)--(8)), a contradiction.
\begin{figure}[h]
	\centering
	\includegraphics[scale=0.69]{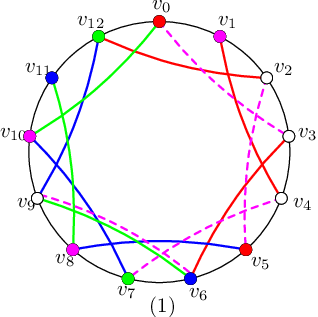}~
    \includegraphics[scale=0.69]{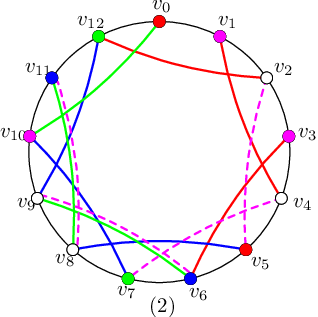}~      
    \includegraphics[scale=0.69]{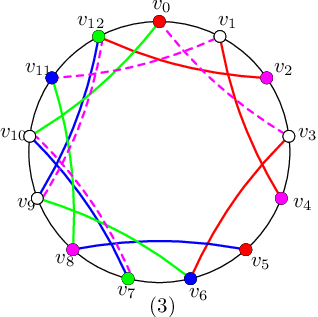}~
	\includegraphics[scale=0.69]{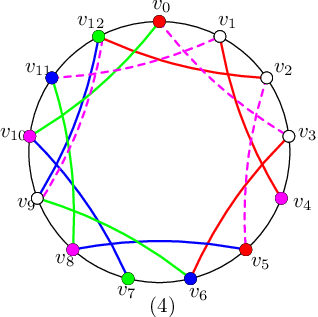}
 
     \vspace{7pt} 
    \includegraphics[scale=0.69]{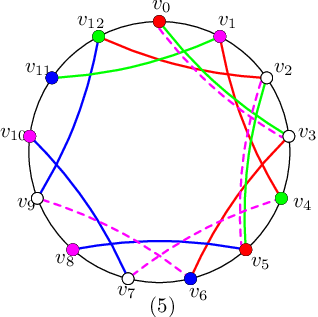}~
    \includegraphics[scale=0.69]{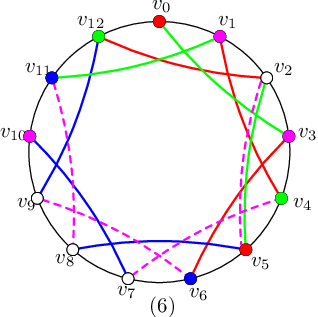}~
     \includegraphics[scale=0.69]{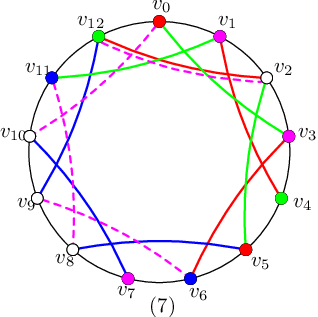}~
	\includegraphics[scale=0.69]{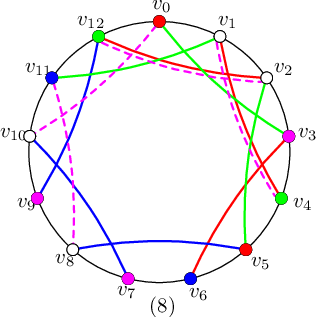}
		\caption{ $\sigma(C_{13}(1, 3))$ for  $d^{2}_{1,2} =6$.}
	\label{Fign24}
\end{figure}

From Cases 1--3, the assumption that  $C_{13}(1,3)$ has a total $5$-colouring does not hold.  Then there is   $\chi''(C_{13}(1,3)) \geq 6$.
  However, according to Lemma \ref{Lem1T}, $\chi''(C_{13}(1,3)) \leq 6$. So,
 $\chi''(C_{13}(1,3)) = 6$.
\end{proof}

\section{Conclusion }\label{Sec4}

In conclusion, we have completely determined the  total  chromatic numbers of $C_{n}(1, 3)$ for all $n\geq7$.
By combining Lemmas \ref{Lempq}, \ref{Lem51}, \ref{Lem7}--\ref{Lem22}, \ref{Lem8} and \ref{Lem13},   we  obtain the following theorem.

\begin{theorem}
$\chi''(C_{n}(1,3)) = 6$ for $n=7,8,12,13,17$,  and 5 for all others.
\label{The11}
\end{theorem}

 In other words, 4-regular circulant graphs $C_{n}(1, 3)$ are Type II  for n=$7,8,12,13,17$, and  Type I for all others. 

\Acknowledgements{}{We would like to thank  Bojan Mohar and Levit Maxwell for their helpful discussions.}
%
%
%
%


\begin{thebibliography}{99}


\bibitem{Be65}
M. Behzad, {\em Graphs and their chromatic numbers}, PhD thesis, Michigan State University, East Lansing, USA, 1965. 

\bibitem{Beh67}
M. Behzad, G. Chartrand and J.K. Cooper, {\it The colour numbers of complete graphs}, J. London Math. Soc. \textbf{42} (1), 226--228, 1967.

\bibitem{Cam03}
 C.N. Campos and C.P. de Mello, {\it Total colouring of $C^2_n$}, Trends Appl. Comput. Math. \textbf{4} (2), 177--186, 2003. 

\bibitem{Che88}
A.G. Chetwynd and A.J.W. Hilton, {\it Some refinements of the total chromatic number conjecture}, Congr. Numer. \textbf{66}, 195--216, 1988.

\bibitem{Far23}
 L. Faria, M. Nigro, M. Preissmann and D. Sasaki, {\it Results about the total chromatic number and the conformability of some families of circulant graphs}, Discrete Appl. Math. \textbf{340}, 123--133, 2023. 

\bibitem{Gee21}
J. Geetha, K. Somasundaram and H.L. Fu, {\it Total coloring of circulant graphs}, Discrete Math. Algorithms Appl. \textbf{13} (5), 2150050, 2021.

\bibitem{Gee23}
J. Geetha, N. Narayanan and K. Somasundaram, {\it Total colorings---a survey}, AKCE Int. J. Graphs Combin. \textbf{20} (3), 339--351, 2023.

\bibitem{Khe08}
 R. Khennoufa and O. Togni, {\it Total and fractional total colourings of circulant graphs}, Discrete Math. \textbf{308} (24), 6316--6329, 2008.

\bibitem{Kos77}
A.V. Kostochka, {\it The total chromatic number of a multigraph with maximal degree 4}, Discrete Math. \textbf{17}, 161--163, 1977. 

\bibitem{Mc94}
C.J.H. McDiarmid and A. Sánchez-Arroyo, {\it Total coloring regular bipartite graphs is NP-hard}, Discrete Math. \textbf{124}, 155--162, 1994. 

\bibitem{Nav22}
R. Navaneeth, J. Geetha, K. Somasundaram and H.L. Fu, {\it Total colorings of some classes of four regular circulant graphs}, AKCE Int. J. Graphs Combin. \textbf{21} (1), 1--3, 2024. 

\bibitem{Nig21}
M. Nigro, M.N. Adauto and D. Sasaki, {\it On total coloring of 4-regular circulant graphs}, Procedia Comput. Sci. \textbf{195}, 315--324, 2021. 

\bibitem{Pal23}
M.A.D.R. Palma, I.F.A. Gonçalves, D. Sasaki and S. Dantas, {\it On total coloring and equitable total coloring of infinite snark families}, RAIRO - Oper. Res. \textbf{57}, 2619--2637, 2023. 

\bibitem{Pra22}
S. Prajnanaswaroopa, J. Jayabalan, K. Somasundaram, H.L. Fu and N. Narayanan, {\it On total coloring of some classes of regular graphs}, Taiwanese J. Math. \textbf{26} (4), 667--683, 2022.

\bibitem{SArr89}
A. Sánchez-Arroyo, {\it Determining the total coloring number is NP-hard}, Discrete Math. \textbf{78}, 315--319, 1989. 

\bibitem{Tong19}
C.L. Tong, X.H. Lin and Y.S. Yang, {\it Equitable total coloring of generalized Petersen graphs $P(n,k)$}, Ars Combin. \textbf{143}, 321--336, 2019.

\bibitem{Tong09}
C.L. Tong, X.H. Lin, Y.S. Yang and Z.H. Li, {\it Equitable total coloring of $C_m \Box C_n$}, Discrete Appl. Math. \textbf{157}, 596--601, 2009. 

\bibitem{Yap96}
H.P. Yap, {\em Total colourings of graphs}, Lecture Notes in Mathematics \textbf{1623}, Springer-Verlag, Berlin, 1996. 

\bibitem{Viz68}
V.G. Vizing, {\it Some unsolved problems in graph theory}, Russian Math. Surveys \textbf{23} (6), 125--141, 1968.
\end{thebibliography}
\end{document}